\documentclass{birkjour}
 \newtheorem{thm}{Theorem}[section]
 \newtheorem{cor}[thm]{Corollary}
 \newtheorem{lem}[thm]{Lemma}
 \newtheorem{prop}[thm]{Proposition}
 \theoremstyle{definition}
 
 \theoremstyle{remark}
 \newtheorem{rem}[thm]{Remark}
 
 \numberwithin{equation}{section}

\begin{document}

%
%
%
%
%
%
%
%
%

\title[Certain results on almost contact pseudo-metric
manifolds]{Certain Results on Almost Contact Pseudo-Metric
Manifolds}

\author[Venkatesha]{Venkatesha}

\address{Department of Mathematics\br
Kuvempu University\br
Shankaraghatta - 577 451\br
Karnataka, INDIA.}
\email{vensmath@gmail.com}


\author[D. M. Naik]{Devaraja Mallesha Naik}
\address{Department of Mathematics\br
    Kuvempu University\br
    Shankaraghatta - 577 451\br
    Karnataka, INDIA.}
\email{devarajamaths@gmail.com}

\author[M. M. Tripathi]{Mukut Mani Tripathi}
\address{Department of Mathematics\br
    Institute of Science\br
    Banaras Hindu University\br
    Varanasi 221005, INDIA.}
\email{mmtripathi66@yahoo.com}




\begin{abstract}
We study the geometry of almost contact pseudo-metric manifolds in
terms of tensor fields $h:=\frac{1}{2}\pounds _\xi \varphi$ and
$\ell := R(\cdot,\xi)\xi$, emphasizing analogies and differences
with respect to the contact metric case. Certain identities
involving $\xi$-sectional curvatures are obtained. We establish
necessary and sufficient condition for a nondegenerate almost $CR$
structure $(\mathcal{H}(M), J, \theta)$ corresponding to almost
contact pseudo-metric manifold $M$ to be $CR$ manifold. Finally, we
prove that a contact pseudo-metric manifold $(M,
\varphi,\xi,\eta,g)$ is Sasakian if and only if the
corresponding nondegenerate almost $CR$ structure $(\mathcal{H}(M),
J)$ is integrable and $J$ is parallel along $\xi$ with respect to
the Bott partial connection.
\end{abstract}

\maketitle

\section{Introduction}
In $1969$, Takahashi \cite{TT} initiated the study of contact
structures associated with pseudo-Riemannian metrics. Afterwards, a
number of authors studied such structures mainly focusing on a
special case, namely Sasakian pseudo-metric manifolds. The case of
contact Lorentzian structures $(\eta, g)$, where $\eta$ is a contact
$1$-form and $g$ a Lorentzian metric associated to it, has a
particular relevance for physics and was considered in \cite{Dug}
and \cite{BejDug}. A systematic study of almost contact
pseudo-metric manifolds was undertaken by Calvaruso and Perrone
\cite{CalPer1} in 2010, introducing all the technical apparatus
which is needed for further investigations, and such manifolds have
been extensively studied under several points of view in
\cite{AP1,Calv,AP2,CalPer3,Per1,Per2,Per3,Wang,AalCar,CarrPer}, and
references cited therein.

The operators $h:=\frac{1}{2}\pounds_\xi \varphi$ and $\ell : =
R(\cdot,\xi)\xi$ play fundamental roles in the study of geometry of
contact pseudo-metric manifolds. For contact metric manifolds,
Sharma \cite{RS} obtained the following beautiful results (Theorem
1.1 in \cite{RS}):
\begin{itemize}
\item [(a)] a contact metric manifold is $K$-contact if and only if
$h$ is a Codazzi tensor;
\item [(b)] a contact metric manifold is $K$-contact if and only
if $\tau$, the tensor metrically equivalent to the strain tensor
$\pounds_\xi g$ of $M$ along $\xi$, is a Codazzi tensor;
\item [(c)]  the sectional curvatures of all plane
sections containing $\xi$ vanish if and only if the tensor $\ell$ is
parallel.
\end{itemize}
The proof of these results exploit, in an essential way, the fact
that in the contact Riemannian case, the self-adjoint operator $h$
vanishes if $h^2=0$. But in the contact pseudo-metric case the
condition $h^2=0$ does not necessarily imply that $h=0$ (see
\cite{Per2}). So the corresponding results fail for general contact
pseudo-metric structures.

Under these circumstances, becomes interesting to explore more the
geometry of contact pseudo-metric manifolds. The paper is organized
as follows. In section~\ref{S:02}, we give the basics of almost
contact pseudo-metric manifolds. In section~\ref{S:03}, we study
contact pseudo-metric manifold $M$ with $h$ satisfying Codazzi
condition and we prove that $M$ is Sasakian pseudo-metric manifold
if and only if the equation \eqref{E:02.4} is satisfied and $h$ is a
Codazzi tensor. In Section~\ref{S:04}, we investigate the Codazzi
condition for the operator $\tau$, and we obtain a necessary and
sufficient condition for $\tau$ to be a Codazzi tensor on contact
pseudo-metric manifold. Moreover, if $\tau$ is a Codazzi tensor,
then $h^2=0$ and the Ricci operator $Q$ satisfies $Q\xi=2\varepsilon
n\xi$, and we prove that $M$ is a Sasakian pseudo-metric manifold if
and only if the equation \eqref{E:02.4} is satisfied and $\tau$ is a
Codazzi tensor. In section~\ref{S:05}, we obtain certain identities
involving $\xi$-sectional curvatures of contact pseudo-metric
manifolds. It is proved that the parallelism of the tensor $\ell$
together with the condition $\nabla_\xi h=0$ on a contact
pseudo-metric manifold implies that all $\xi$-sectional curvatures
vanish. At the end, we investigate the nondegenerate almost $C\!R$
structure $(\mathcal{H}(M), J, \theta)$ corresponding to almost
contact pseudo-metric manifold $M$, and establish a necessary and
sufficient condition for an almost contact pseudo-metric manifold to
be a $C\!R$ manifold. Finally, we show that a contact pseudo-metric
manifold $(M, \varphi,\xi,\eta,g)$ is Sasakian pseudo-metric if and
only if the corresponding nondegenerate almost $C\!R$ structure
$(\mathcal{H}(M),J)$ is integrable and $J$ is parallel along $\xi$
with respect to the Bott partial connection.


\section{Preliminaries}\label{S:02}
In this section, we briefly recall some general definitions and
basic properties of almost contact pseudo-metric manifolds. For more
information and details, we recommend the reference \cite{CalPer1}.

A $(2n+1)$-dimensional smooth connected manifold $M$ is said to be
an \textit{almost contact manifold} if there exists on $M$ a $(1,1)$
tensor field $\varphi$, a vector field $\xi$, and a $1$-form $\eta$
such that
\begin{equation} \label{E:e1}
\varphi^2 = -I + \eta\otimes\xi, \quad \eta (\xi) = 1, \quad
\varphi\xi=0, \quad \eta\circ\varphi=0
\end{equation}
for all $X,Y\in TM$. It is known that the first relation along with
any one of the remaining three relations in \eqref{E:e1} imply the
remaining two relations. Also, for an almost contact structure, the
rank of $\varphi$ is $2n$. For more details, we refer to
\cite{BDE2}.

\par If an almost contact manifold is endowed with a
pseudo-Riemannian metric $g$ such that
\begin{equation}\label{E:e3}
g(\varphi X, \varphi Y)=g( X, Y)-\varepsilon\eta(X)\eta(Y),
\end{equation}
where $\varepsilon=\pm 1$, for all $X,Y\in TM$, then $(M, \varphi,
\xi, \eta, g)$ is called an \textit{almost contact pseudo-metric
manifold}. The relation \eqref{E:e3} is equivalent to
\begin{equation}\label{E:e3a}
\eta(X)=\varepsilon g(X, \xi)\; {\rm along} \; {\rm with}\;
g(\varphi X,Y)=-g(X,\varphi Y).
\end{equation}
In particular, in an almost contact pseudo-metric manifold, it
follows that $g(\xi, \xi)=\varepsilon$ and so, the characteristic
vector field $\xi$ is a unit vector field, which is either
space-like or time-like, but cannot be light-like.

\par The \textit{fundamental $2$-form} of an almost contact
pseudo-metric manifold $(M, \varphi, \xi, \eta, g)$ is defined by
\[
\Phi(X,Y)=g(X, \varphi Y),
\]
which satisfies $\eta \wedge \Phi^{n}\neq 0$. An almost contact
pseudo-metric manifold is said to be a \textit{contact pseudo-metric
manifold} if $d\eta=\Phi$, where
\[
d\eta(X, Y)=\frac{1}{2}(X\eta(Y)-Y\eta(X)-\eta([X, Y])).
\]
The curvature operator $R$ is given by
\[
R(X, Y)=[\nabla_X, \nabla_Y]-\nabla_{[X, Y]}.
\]
This sign convention of $R$ is opposite to the one used in
\cite{CalPer1,CalPer3,Per1,Per2,Per3}. The Ricci operator $Q$ is
determined by
\[
S(X, Y ) = g(QX, Y ).
\]
In an almost contact pseudo-metric manifold $(M, \varphi, \xi, \eta,
g)$ there always exists a special kind of local pseudo-orthonormal
basis $\{e_i, \varphi e_i, \xi\}_{i=1}^n$, called a local
$\varphi$-basis.

\par In a contact pseudo-metric manifold, the $(1,1)$ tensor
$h=\frac{1}{2}\pounds_\xi \varphi$ is self-adjoint and satisfies
\[
h\xi=0,\quad  \varphi h + h\varphi =0, \quad \text{tr}(h) =
\text{tr}(\varphi h)=0.
\]
Further, one has the following formulas:
\begin{gather}
\nabla_X \xi=-\varepsilon\varphi X-\varphi h X,\label{E:0:2.3}\\
(\pounds_\xi g)(X, Y)=2g(h\varphi X, Y),\label{E:0:2.4}\\
(\nabla_\xi h)X=\varphi X-h^2\varphi X+\varphi R(\xi, X)\xi,\label{E:0:2.5}\\
R(\xi, X)\xi-\varphi R(\xi,\varphi
X)\xi=2(h^2+\varphi^2)X,\label{E:0:2.6} \\ \text{tr }\nabla
\varphi=2n\xi.\label{E:0:2.7}
\end{gather}

\par A contact pseudo-metric manifold $M$ is said to be
a \textit{$K$-contact pseudo-metric manifold} if $\xi$ is a Killing
vector field (or equivalently, $h=0$), and is said to be a
\textit{Sasakian pseudo-metric manifold} if the almost complex
structure $J$ on the product manifold $M\times {\Bbb R}$ defined by
\[
J\left(X, f\frac{d}{dt}\right)=\left(\varphi X - f\xi,
\eta(X)\frac{d}{dt}\right),
\]
is integrable, where $X\in TM$, $t$ is the coordinate on ${\Bbb R}$
and $f$ is a $C^\infty$ function on $M\times {\Bbb R}$. It is well
known that a contact pseudo-metric manifold $M$ is a Sasakian
pseudo-metric manifold if and only if
\begin{equation}\label{E:00:2.8}
(\nabla_X\varphi)Y=g(X,Y)\xi - \varepsilon\eta(Y)X
\end{equation}
for all $X, Y\in TM$. A Sasakian pseudo-metric manifold is always
$K$-contact pseudo-metric. A 3-dimensional $K$-contact pseudo-metric
manifold becomes a Sasakian pseudo-metric manifold, which may not be
true in higher dimensions. Further on a Sasakian pseudo-metric
manifold we have
\begin{equation}\label{E:02.4}
R(X,Y)\xi=\eta(Y)X-\eta(X)Y.
\end{equation}
In contact metric case, the condition \eqref{E:02.4} implies that
the manifold is Sasakian, which is not true in contact pseudo-metric
case \cite{Per1}. However, we have the following:
\begin{lem}\label{L:0:2.1} \cite{Per1}
Let $M$ be a $K$-contact pseudo-metric manifold. Then $M$ is a
Sasakian pseudo-metric manifold if and only if the curvature tensor
$R$ satisfies \eqref{E:02.4}.
\end{lem}


\section{The Codazzi condition for $h$}\label{S:03}
A self-adjoint tensor $A$ of type $(1,1)$ on a pseudo-Riemannian
manifold is known to be a Codazzi tensor if
\begin{equation}
(\nabla_X A)Y=(\nabla_Y A)X
\end{equation}
for all $X, Y\in TM$. Now, we prove the following:

\begin{thm}\label{T:3.1}
Let $M$ be a  contact pseudo-metric manifold. Then the following
statements are true:
\begin{itemize}
\item [(a)] If $h$ is a Codazzi tensor, then $h^2=0$.
\item [(b)] $M$ is a Sasakian pseudo-metric manifold if and only if
$M$ satisfies \eqref{E:02.4} and $h$ is a Codazzi tensor.
\end{itemize}
\end{thm}
\begin{proof}
(a). Suppose that $h$ is a Codazzi tensor, that is,
\[
(\nabla_X h)Y=(\nabla_Yh)X, \qquad X,Y \in TM.
\]
For $Y=\xi$, using \eqref{E:0:2.3} in the above equation, we obtain
\begin{equation*}
(\nabla_\xi h)X=-\varepsilon \varphi h X- h^2 \varphi X.
\end{equation*}
In view of \eqref{E:0:2.5}, the above equation turns into
\begin{equation}\label{E:0:3.2}
\varphi R(\xi, X)\xi=-\varepsilon \varphi h X- \varphi X.
\end{equation}
Operating $\varphi$ on both sides of \eqref{E:0:3.2}, it follows
that
\begin{equation}\label{E:0:3.3}
R(\xi, X)\xi=\varphi^2X-\varepsilon h X.
\end{equation}
Making use of \eqref{E:0:3.3} in \eqref{E:0:2.6}, shows that $h^2=0$.

\smallskip

(b). If $M$ is a Sasakian pseudo-metric manifold, then $h=0$ and $M$
satisfies \eqref{E:02.4}; and the result is trivial. Conversely,
suppose that \eqref{E:02.4} is true and $h$ is a Codazzi tensor.
From \eqref{E:02.4}, we obtain that
\begin{equation}\label{E:0:3.4}
R(\xi, X)\xi=\varphi^2 X, \qquad X \in TM.
\end{equation}
Equations \eqref{E:0:3.3} and \eqref{E:0:3.4} imply that $h=0$, that
is, $M$ is a $K$-contact pseudo-metric manifold. Thus, the result
follows from Lemma~\ref{L:0:2.1}.
\end{proof}

\begin{rem}
In a contact Riemannian manifold, if $h$ is a Codazzi tensor, then
$h=0$, that is, the manifold becomes $K$-contact manifold \cite{RS}.
In the Riemannian case, as $h^2=0$ implies $h=0$,
Theorem~\ref{T:3.1}~(a) holds in a stronger form, that is, $M$ is
$K$-contact if and only if $h$ is a Codazzi tensor. But, in the case
of $M$ being contact pseudo-metric, the condition $h^2=0$ does not
imply that $h=0$, because $h$ may not be diagonalizable (see
\cite{Per2}). Note that the result (b) of
Theorem~\ref{T:3.1} is stronger than the Lemma~\ref{L:0:2.1} which
was proved in \cite{Per1}.
\end{rem}

In a contact Lorentzian manifold, just like the case of contact
metric manifold, the condition $h^2=0$ implies $h=0$ (see
\cite{Calv}). Hence, we immediately have the following

\begin{cor}
Let $M$ be a contact Lorentzian manifold. If $h$ is a Codazzi
tensor, then $h=0$, that is, $M$ is $K$-contact Lorentzian manifold.
\end{cor}

\section{The Codazzi condition for $\tau$}\label{S:04}

We denote by $\tau$, the tensor metrically equivalent to the strain
tensor $\pounds_\xi g$ along $\xi$, that is,
\[
g(\tau X, Y)=(\pounds_\xi g)(X,Y)
\]
for all $X, Y\in TM$. As pointed out in the introduction, in a
contact metric manifold, if $\tau$ satisfies the Codazzi condition,
then $h=0$, that is, the manifold is a $K$-contact manifold. This
fact need not be true in the case of contact pseudo-metric
manifolds. So, it is quite interesting to study contact
pseudo-metric manifolds, which satisfy the Codazzi condition for
$\tau$. Now we prove the following:
\begin{lem}
In a contact pseudo-metric manifold, $\tau$ is a Codazzi tensor if
and only if the curvature tensor $R$ satisfies
\begin{equation}\label{E:0:4.1}
R(\xi, X)Y=\varepsilon (\nabla_X \varphi)Y.
\end{equation}
\end{lem}

\begin{proof}
Treating $\nabla \xi$ as a tensor of type $(1,1)$, that is
$\nabla\xi : X \mapsto \nabla_X \xi$, one can see that
\begin{equation*}
R(X, Y)\xi=(\nabla_X\nabla\xi)Y-(\nabla_Y\nabla\xi)X,
\end{equation*}
which together with \eqref{E:0:2.3} gives
\begin{equation}\label{E:0:4.2}
R(X,Y)\xi=-\varepsilon(\nabla_X\varphi)Y-(\nabla_X\varphi
h)Y+\varepsilon(\nabla_Y\varphi)X+(\nabla_Y\varphi h)Y.
\end{equation}
On the other hand, if $\tau$ is a Codazzi tensor, then from
\eqref{E:0:2.4} we have
\begin{equation*}
(\nabla_X h\varphi)Y=(\nabla_Y h\varphi)X.
\end{equation*}
Thus, \eqref{E:0:4.2} shows that $\tau$ is a Codazzi tensor if and
only if
\begin{equation}\label{E:0:4.3}
R(X,Y)\xi=\varepsilon\{(\nabla_Y\varphi)X-(\nabla_X\varphi)Y\}.
\end{equation}
Now if $\tau$ is a Codazzi tensor, then by using Bianchi identity
and \eqref{E:0:4.3}, we get
\begin{align*}
R(\xi, X, Y,Z)&=\varepsilon\{g(X, (\nabla_X\varphi)Y)-g(Y,
(\nabla_X\varphi)Z)+g(Z, (\nabla_X\varphi)Y)\\
&\quad-g(X, (\nabla_Z\varphi)Y)\}\\
&=-2\varepsilon g((\nabla_X\varphi)Z, Y)+R(Z,Y,\xi, X),
\end{align*}
and so
\begin{equation*}
R(\xi,X,Y,Z)=-\varepsilon g((\nabla_X\varphi)Z, Y),
\end{equation*}
which gives \eqref{E:0:4.1}.

Conversely, if \eqref{E:0:4.1} is true, then from Bianchi identity
we have
\begin{align*}
R(X, Y, \xi, Z)&=R(\xi, Z, X, Y)=-R(Z,X,\xi,Y)-R(X, \xi, Z, Y)\\
&=-R(\xi, Y, Z, X)+R(\xi, X, Z, Y)\\
&=-\varepsilon\{g((\nabla_Y\varphi)Z, X)-g((\nabla_X\varphi)Z,Y)\},
\end{align*}
which leads to \eqref{E:0:4.3}, and hence $\tau$ is a Codazzi
tensor.
\end{proof}

\begin{thm}
Let $M$ be a contact pseudo-metric manifold. Then the following
statements are true.
\begin{itemize}
\item [(i)] If $\tau$ is a Codazzi tensor, then $h^2=0$ and the
Ricci operator $Q$ satisfies
\begin{equation}\label{E:0:4.4}
Q\xi=2\varepsilon n\xi.
\end{equation}
\item [(ii)] $M$ is Sasakian if and only if $M$ satisfies
\eqref{E:02.4} and $\tau$ is a Codazzi tensor.
\end{itemize}
\end{thm}

\begin{proof}
(i). If $\tau$ is a Codazzi tensor, then \eqref{E:0:4.1} gives
\begin{equation*}
R(\xi, X)\xi=\varepsilon(\nabla_X\varphi)\xi=\varphi^2X-\varepsilon
h X,
\end{equation*}
where we used \eqref{E:0:2.3}. This implies
\begin{equation*}
\varphi R(\xi, \varphi X)\xi=-\varphi^2X-\varepsilon h X,
\end{equation*}
and so
\begin{equation}\label{E:0:4.5}
R(\xi, X)\xi-\varphi R(\xi, \varphi X)\xi=2\varphi^2X.
\end{equation}
Comparing \eqref{E:0:2.6} and \eqref{E:0:4.5}, we obtain $h^2=0$.

Now, if $\{e_i\}_i^{2n+1}$ is any local pseudo-orthonormal basis,
then considering \eqref{E:0:4.1} we get
\begin{align*}
S(X, \xi)&=\sum_{i=1}^{2n+1}\varepsilon_i R(e_i, X, \xi, e_i) =
\varepsilon\sum_{i=1}^{2n+1}\varepsilon_i
g((\nabla_{e_i}\varphi)e_i, X) \\
&=\varepsilon g(\text{tr}(\nabla \varphi), X),
\end{align*}
which by using \eqref{E:0:2.7} we have \eqref{E:0:4.4}.

\smallskip

(ii). Suppose that $M$ is a Sasakian pseudo-metric manifold, then
$M$ satisfies \eqref{E:02.4} and $h=0$.

Conversely, suppose that $M$ satisfies \eqref{E:02.4} and $\tau$ is
a Codazzi tensor. Then \eqref{E:0:4.1} shows that
\begin{align*}
g((\nabla_X\varphi)Y, Z)&=\varepsilon R(\xi, X, Y, Z)=-\varepsilon
R(Z, Y, \xi, X)\\ &=-\varepsilon\{ \eta(Y)g(Z,X)-\varepsilon g(Z,
\xi)g(X, Y)\}
\end{align*}
which gives \eqref{E:00:2.8}. Hence $M$ becomes a Sasakian
pseudo-metric manifold.
\end{proof}

\begin{cor}
Let $M$ be a contact Lorentzian manifold. If $\tau$ is a Codazzi
tensor, then $h=0$, that is, $M$ is $K$-contact Lorentzian manifold.
\end{cor}


\section{$\xi$-Sectional Curvatures}\label{S:05}

The \textit{$\xi$-sectional curvature} $K(\xi, X )$ of a contact
pseudo-metric manifold is defined by
\[
K(\xi, X)=\varepsilon {\varepsilon_{\! X}} g(R(\xi,X)X, \xi),
\]
where $X$ is a unit vector field such that $X\in \text{Ker}\, \eta$
and $g(X, X)={\varepsilon_{\! X}} =\pm 1$.

It is well known that a contact metric manifold is $K$-contact if
and only if all $\xi$-sectional curvatures are equal to +1 (see
\cite{BDE2}). The corresponding result in pseudo-Riemannian case
need not be true. In fact, we have the following:

\begin{thm}
If $M$ is a $K$-contact pseudo-metric manifold, then all
$\xi$-sectional curvatures are equal to $\varepsilon$.
\end{thm}

\begin{proof}
If $M$ is a $K$-contact pseudo-metric manifold, then $h=0$. So
\eqref{E:0:2.5} becomes
\[
\varphi R(\xi, X)\xi=-\varphi X,
\]
which upon applying $\varphi$ gives
\[
R(\xi, X)\xi = - X
\]
for $X\in \text{Ker}\, \eta$. Thus
\[
K(\xi, X)=\varepsilon {\varepsilon_{\! X}} g(R(\xi,X)X,
\xi)=\varepsilon {\varepsilon_{\! X}} g(X, X)=\varepsilon.
\]
\end{proof}

\begin{rem}
The converse of above result is not true in general. In fact, a
contact pseudo-metric manifold $M$, which satisfies \eqref{E:02.4},
has $\xi$-sectional curvatures equal to $\varepsilon$. But we
already know that the condition \eqref{E:02.4} does not necessarily
imply that $M$ is a $K$-contact pseudo-metric manifold.
\end{rem}

Now we prove the following:

\begin{thm}
On a contact pseudo-metric manifold $M$, the $\xi$-sectional
curvatures satisfy
\begin{gather}
K(\xi, X)=\varepsilon \{1-{\varepsilon_{\! X}} g(h^2X,
X)-{\varepsilon_{\! X}} g((\nabla_\xi h)X, \varphi
X)\},\label{E:0:5.1}\\
K(\xi, X)=K(\xi, \varphi X)-2\varepsilon {\varepsilon_{\! X}}
g((\nabla_\xi h)X , \varphi X)\label{E:0:5.2}
\end{gather}
for any unit vector $X\in \text{Ker}\, \eta$.
\end{thm}

\begin{proof}
Using \eqref{E:0:2.5}, we have
\begin{align*}
K(\xi, X)&=-\varepsilon {\varepsilon_{\! X}} R(\xi, X, \xi, X)\\
&=-\varepsilon {\varepsilon_{\! X}} g(-\varphi(\nabla_\xi
h)X-X+h^2X, X)\\
&=\varepsilon \{{\varepsilon_{\! X}} g(\varphi (\nabla_\xi h)X,
X)+{\varepsilon_{\! X}}^2-{\varepsilon_{\! X}} g(h^2X, X)\},
\end{align*}
which gives \eqref{E:0:5.1}.

Now, plugging $X$ by $\varphi X$ in \eqref{E:0:5.1} keeping $h
\varphi=-\varphi h$ and $\nabla_\xi\varphi=0$ in mind, we obtain
\begin{equation}\label{E:0:5.3}
K(\xi, \varphi X)=\varepsilon \{1-{\varepsilon_{\! X}} g(h^2X,
X)+{\varepsilon_{\! X}} g((\nabla_\xi h)X, \varphi X)\}.
\end{equation}
Now, from \eqref{E:0:5.1} and \eqref{E:0:5.3}, we get
\eqref{E:0:5.2}.
\end{proof}

\begin{thm}
Let $M$ be a contact pseudo-metric manifold with $\nabla_\xi h=0$.
Then  $h^2=0$ if and only if all $\xi$-sectional curvatures are
equal to $\varepsilon$.
\end{thm}

\begin{proof}
Taking the inner product of the unit vector field $X\in \text{Ker}\,
\eta$ with \eqref{E:0:2.6} yields the following formula for
sectional curvatures:
\begin{equation}\label{E:0:5.4}
K(\xi, X)+K(\xi, \varphi X)=2\varepsilon \{ 1-{\varepsilon_{\! X}}
g(h^2 X, X)\}.
\end{equation}
Now, since $\nabla_\xi h=0$, \eqref{E:0:5.2} yields
\begin{equation}\label{E:0:5.5}
K(\xi, X)=K(\xi, \varphi X)
\end{equation}
for any unit vector $X\in \text{Ker } \eta$. From \eqref{E:0:5.4}
and \eqref{E:0:5.5} we see
\begin{equation*}
K(\xi, X)=\varepsilon \text{\; if and only if \;} g(h^2 X, X)=0.
\end{equation*}
This concludes the proof.
\end{proof}

\begin{cor}
A contact Lorentzian manifold is a $K$-contact Lorentzian manifold
if and only if all $\xi$-sectional curvatures are equal to $-1$.
\end{cor}

As we discussed in introduction, due to the fact that $h^2=0$ does
not imply $h=0$ in a contact pseudo-metric manifold, the parallel
condition of $\ell$ does not imply that $\xi$-sectional curvatures
vanish. However, we have the following:

\begin{thm}
If $M$ is a contact pseudo-metric manifold with $\nabla_\xi h=0$ and
$\nabla \ell=0$, then all $\xi$-sectional curvatures vanish.
\end{thm}

\begin{proof}
Applying by $\varphi$ on both sides of \eqref{E:0:2.5} and using
$\nabla_\xi h=0$, it follows that
\begin{equation}\label{E:0:5.6}
\ell X = -h^2 X + X-\eta(X)\xi,
\end{equation}
for any $X\in TM$. Now, in view of $(\nabla_X \ell)\xi=0$ and
\eqref{E:0:5.6}, we have
\begin{equation}\label{E:0:5.7}
\varepsilon h^2 \varphi X-h^3 \varphi X-\varepsilon \varphi X+h
\varphi X=0.
\end{equation}
If  $X\in \text{Ker } \eta$ is a unit vector field, then taking the
inner product of $\varphi X$ with \eqref{E:0:5.7} leads to
\begin{equation}\label{E:0:5.8}
\varepsilon g(h^2 X, X)+g(h^3 X, X)-\varepsilon g(X,X)-g(h X, X)=0.
\end{equation}
Now replacing $X$ by  $\varphi X$ in \eqref{E:0:5.7} and then taking
inner product of $X$ with the resulting equation gives
\begin{equation}\label{E:0:5.9}
-\varepsilon g(h^2 X, X)+g(h^3 X, X)+\varepsilon g(X,X)-g(h X, X)=0.
\end{equation}
Now subtracting \eqref{E:0:5.8} from \eqref{E:0:5.9} yields
\begin{equation}\label{E:0:5.10}
g(h^2 X, X)=g(X,X)={\varepsilon_{\! X}}
\end{equation}
for any unit vector $X\in \text{Ker}\, \eta$. Using \eqref{E:0:5.10}
and $\nabla_\xi h=0$ in \eqref{E:0:5.1} we conclude that $K(\xi,
X)=0$.
\end{proof}


\section{Almost $C\!R$ Structures}\label{S:06}

First, we recall few notions of almost $C\!R$ structures (see
\cite{Drag,Per2,Per4}). Let $M$ be a $(2n+1)$-dimensional
(connected) differentiable manifold. Let $\mathcal{H}(M)$ be a
smooth real subbundle of rank $2n$ of the tangent bundle $TM$ (also
called Levi distribution), and $J:\mathcal{H}(M)\to \mathcal{H}(M)$
be a smooth bundle isomorphism such that $J^2=-I$. Then the pair
$(\mathcal{H}(M), J)$ is called an almost $C\!R$ structure on $M$.
An almost $C\!R$ structure is called a $C\!R$ structure if it is
integrable, that is, the following two conditions are satisfied
\begin{gather}
[JX,Y]+[X,JX]\in \mathcal{H}(M),\label{E:6.1}\\
J([JX,Y]+[X,JY])=[JX,JY]-[X,Y]\label{E:6.2}
\end{gather}
for all $X,Y\in \mathcal{H}(M)$.

On an almost $C\!R$ manifold $(M,\mathcal{H}(M),J)$, we define a
$1$-form $\theta$ such that $\text{Ker}\, \theta=\mathcal{H}(M)$, and such a differential $1$-form $\theta$ is called a pseudo-Hermitian structure on $M$. Then on $\mathcal{H}(M)$, the Levi form $L_\theta$ is defined by
\begin{equation*}
L_\theta(X,Y)=d\theta(X,JY)
\end{equation*}
for all $X,Y\in \mathcal{H}(M)$. Furthermore, we define a
$(0,2)$-tensor field on $\mathcal{H}(M)$ by
\begin{equation*}
\alpha(X,Y)=(\nabla_X\theta)(JY)+(\nabla_{JX}\theta)(Y)
\end{equation*}
for all $X,Y\in\mathcal{H}(M)$.

Then we have the following:

\begin{prop}
For an almost $C\!R$ structure $(\mathcal{H}(M),J,\theta)$, the
following statements are equivalent\/{\rm :}
\begin{itemize}
\item[(i)] $L_\theta$ is Hermitian, that is, $L_\theta(JX,JY)=L_\theta(X,Y)$;
\item[(ii)] $L_\theta$ is symmetric, that is, $L_\theta(X,Y)=L_\theta(Y,X)$;
\item[(iii)] $[JX,Y]+[X,JY]\in \mathcal{H}(M)$;
\item[(iv)]  $\alpha$ is symmetric, that is, $\alpha(X,Y)=\alpha(Y,X)$.
\end{itemize}
\end{prop}

\begin{proof}
It is immediate that (i)$\Leftrightarrow$(ii) and
(ii)$\Leftrightarrow$(iii) follows from the fact that
\[
d\theta(X,Y)=-\frac{1}{2}\theta([X,Y])
\]
for all $X,Y\in\mathcal{H}(M)$. On the other hand, as in general
\begin{equation*}
d\theta(X,Y)=\frac{1}{2}((\nabla_X\theta)Y-(\nabla_Y\theta)X),
\end{equation*}
the condition (ii) is equivalent to
\[
(\nabla_X\theta)(JY)+(\nabla_{JX}\theta)Y=(\nabla_Y\theta)(JX)+
(\nabla_{JY}\theta)X,
\]
and so (ii)$\Leftrightarrow$(iv).
\end{proof}

An almost pseudo-Hermitian $C\!R$
structure $(\mathcal{H}(M), J,\theta)$ is said to be nondegenerate
if the Levi form $L_\theta$ is a nondegenerate Hermitian form, and
so the $1$-form $\theta$ is a contact form.

Let $(M, \mathcal{H}(M),J,\theta)$ be a nondegenerate
pseudo-Hermitian almost $C\!R$ manifold. We extend the complex
structure $J$ to an endomorphism $\varphi$ of the tangent bundle
$TM$ in such a way that $\theta=J$ on $\mathcal{H}(M)$ and $\varphi
\xi=0$, where $\xi$ is the Reeb vector field of $\theta$. Then the
Webster metric $g_\theta$, which is a pseudo-Riemannian metric, is
defined by
\begin{equation*}
g_\theta(X,Y)=L_\theta(X,Y),\quad g_\theta(X,\xi)=0, \quad
g_\theta(\xi,\xi)=\varepsilon
\end{equation*}
for all $X, Y\in \mathcal{H}(M)$. In this case,
$(\varphi,\xi,\eta=-\theta, g=g_\theta)$ defines a contact
pseudo-metric structure on $M$. Conversely, if
$(\varphi,\xi,\eta,g)$ is a contact pseudo-metric structure, then
$(\mathcal{H}(M), J, \theta)$, where $\mathcal{H}(M)=\text{Ker}\,
\eta$, $\theta=-\eta$, and $J=\varphi_{\arrowvert\mathcal{H}(M)}$,
defines a nondegenerate almost $C\!R$ structure on $M$. Thus, we
have:

\begin{prop}[\cite{Per2}]
The notion of nondegenerate almost $C\!R$ structure
$(\mathcal{H}(M), J, \theta)$ is equivalent to the notion of contact
pseudo-metric structure $(\varphi,\xi,\eta,g)$.
\end{prop}

Now, we prove the following:

\begin{thm}
The nondegenerate almost $C\!R$ structure $(\mathcal{H}(M), J,
\theta)$ corresponding to almost contact pseudo-metric manifold $M$
is a $C\!R$ manifold if and only if
\begin{equation}\label{E:6.3}
(\nabla_X J)Y-(\nabla_{JX}J)JY=\alpha(X,Y)\xi
\end{equation}
for all $X,Y\in\mathcal{H}(M)$.
\end{thm}

\begin{proof}
Applying $J$ to \eqref{E:6.2} gives
\begin{equation*}
(\nabla_YJ)X-(\nabla_XJ)Y=J(\nabla_{JX}J)Y-J(\nabla_{JY}J)X
\end{equation*}
for all $X,Y\in\mathcal{H}(M)$. Since
$J(\nabla_{JX}J)Y=-(\nabla_{JX}J)JY$, the above equation becomes
\begin{equation}\label{E:6.4}
(\nabla_YJ)X-(\nabla_XJ)Y=(\nabla_{JY}J)JX-(\nabla_{JX}J)JY
\end{equation}
for all $X,Y\in\mathcal{H}(M)$. If we define a (0,3)-tensor field
$A$ on $\mathcal{H}(M)$ as
\begin{equation}\label{E:6.5}
A(X,Y,Z)=g((\nabla_{JX}J)JY-(\nabla_XJ)Y,Z)
\end{equation}
for all $X,Y\in\mathcal{H}(M)$, then from \eqref{E:6.4} one obtain
\begin{equation}\label{E:6.6}
A(X,Y,Z)=A(Y,X,Z).
\end{equation}
Next, a simple computation shows that
\begin{align*}
A&(X,Y,Z)+A(X,Z,Y)\\
&=g((\nabla_{JX}J)JY-(\nabla_XJ)Y,Z)+g((\nabla_{JX}J)JZ -
(\nabla_XJ)Z,Y) \\
&=-g((\nabla_{JX}J)Z,JY)+g((\nabla_{JX}J)JZ,Y)\\
&=-g(\nabla_{JX}JZ,JY)-g((\nabla_{JX}Z),J^2Y) \\
&\quad+g(\nabla_{JX}J^2Z,Y) - g(J(\nabla_{JX}JZ),Y)\\
&=0,
\end{align*}
where the skew-symmetry of $J$ and $\nabla J$ are used. This
together with \eqref{E:6.6} gives the following:
\begin{align*}
A(X,Y,Z)&=-A(X,Z,Y)=-A(Z,X,Y)=A(Z,Y,X)\\
&=A(Y,Z,X)=-A(Y,X,Z)=-A(X,Y,Z).
\end{align*}
Hence it follows that $A=0$, and so \eqref{E:6.5} implies
\begin{equation}\label{E:6.7}
(\nabla_{JX}J)JY-(\nabla_XJ)Y=\gamma(X,Y)\xi
\end{equation}
for all $X,Y\in\mathcal{H}(M)$, for certain (0,2)-tensor field
$\gamma$ on $\mathcal{H}(M)$. It remains to show that
$\gamma=\alpha$. From \eqref{E:6.7}, it follows that
\begin{align*}
\gamma(X,Y)&=\varepsilon g((\nabla_{JX}J)JY-(\nabla_XJ)Y,\xi)\\
&=\varepsilon\{- g((\nabla_{JX}J)\xi,JY)+g((\nabla_XJ)\xi,Y)\}\\
&=\varepsilon \{ g(\nabla_{JX}\xi,Y)-g(J\nabla_X\xi,Y)\}\\
&=(\nabla_{JX}\theta)Y+(\nabla_X\theta)JY\\
&=\alpha(X,Y).
\end{align*}
Conversely, suppose that \eqref{E:6.3} holds true. Then projecting
\eqref{E:6.3} onto $\xi$, it follows that $\alpha$ is symmetric and
is equivalent to \eqref{E:6.1}. The symmetry of $\alpha$ together
with \eqref{E:6.3} gives \eqref{E:6.4}, which yields
\begin{equation*}
-[JX,Y]-[X,JY])=J[JX,JY]-J[X,Y],
\end{equation*}
for all $X,Y\in\mathcal{H}(M)$, and so satisfies the equation
\eqref{E:6.2}.
\end{proof}

Let $(M, \varphi,\xi,\eta,g)$ be an almost contact pseudo-metric
manifold with $(\mathcal{H}(M), J)$ as the corresponding almost
$C\!R$ structure. For $Y\in TM$, we denote $Y_{\arrowvert
\mathcal{H}(M)}$ to the orthogonal projection on $\mathcal{H}(M)$.
Then, the Bott partial connection $\breve{\nabla}$ on
$\mathcal{H}(M)$ (along $\xi$) is the map
$\breve{\nabla}:S(\xi)\times \mathcal{H}(M)\to \mathcal{H}(M)$
defined by
\begin{equation*}
\breve{\nabla}_\xi X := (\pounds_\xi X)_{\arrowvert
\mathcal{H}(M)}=[\xi, X]_{\arrowvert \mathcal{H}(M)}
\end{equation*}
for any $X\in \mathcal{H}(M)$ (see, \cite[p. 18]{Rov}), where $S(\xi)$
is the 1-dimensional linear subspace of $TM$ generated by $\xi$.

\begin{thm}
Let $(M, \varphi,\xi,\eta,g)$ be an almost contact pseudo-metric
manifold, and $\xi$ a geodesic vector field. Then $h=0$ if and only
if $\breve{\nabla}_\xi J=0$.
\end{thm}

\begin{proof}
As $h\xi=0$, we may observe that, $h=0$ if and only if $hX=0$ for
any $X\in \mathcal{H}(M)$.

Now using $\nabla_\xi \xi=0$, for any $X\in \mathcal{H}(M)$, we have
\begin{equation*}
\eta([\xi,X])=\varepsilon g(\xi, \nabla_\xi X-\nabla_X \xi)=0,
\end{equation*}
which means $\pounds_\xi X \in \mathcal{H}(M)$. Thus, we get
\begin{align*}
2hX&=\pounds_\xi (\varphi X)-\varphi(\pounds_\xi
X)=\breve{\nabla}_\xi (\varphi X)- \varphi (\breve{\nabla}_\xi X) \\
&=\breve{\nabla}_\xi(JX)-J(\breve{\nabla}_\xi X)=(\breve{\nabla}_\xi
J)X
\end{align*}
for any $X\in \mathcal{H}(M)$, completing the proof.
\end{proof}

For contact pseudo-metric manifold, the structure vector field is
geodesic. So we have the following:

\begin{cor}
A contact pseudo-metric manifold is $K$-contact if and only if
$\breve{\nabla}_\xi J=0$.
\end{cor}

\begin{thm}
A contact pseudo-metric manifold $(M, \varphi,\xi,\eta,g)$ is
Sasakian if and only if the corresponding nondegenerate almost
$C\!R$ structure $(\mathcal{H}(M), J)$ is integrable and
$\breve{\nabla}_\xi J=0$.
\end{thm}

\begin{proof}
First we observe that, following the same proof given in
\cite{Tanno} for the Riemannian case, the integrable condition (that
is, \eqref{E:6.1} and \eqref{E:6.2}) of the corresponding $C\!R$
structure $(\mathcal{H}(M), J)$ is equivalent to
\begin{equation}\label{E:6.8}
(\nabla_X \varphi)Y=-\{(\nabla_X\eta)\varphi
Y\}\xi-\eta(X)\varphi(\nabla_X \xi),
\end{equation}
where
\begin{equation*}
(\nabla_X\eta)\varphi Y=-g(X,Y)+\varepsilon
\eta(X)\eta(Y)-\varepsilon g(hX,Y),
\end{equation*}
and
\begin{equation*}
\varphi(\nabla_X \xi)=\varepsilon X-\varepsilon \eta(X)\xi+hX.
\end{equation*}
Thus, \eqref{E:6.8} becomes
\begin{equation}\label{E:6.9}
(\nabla_X \varphi)Y=g(X+\varepsilon hX,Y)\xi-\varepsilon
\eta(Y)(X+\varepsilon hX).
\end{equation}
If the contact pseudo-metric manifold $(M, \varphi,\xi,\eta,g)$ is
Sasakian, then \eqref{E:6.9} satisfies with $h=0$, and so
corresponding nondegenerate almost $C\!R$ structure
$(\mathcal{H}(M), J)$ is integrable and $\breve{\nabla}_\xi J=0$.

Conversely, as $\breve{\nabla}_\xi J=0$ implies $h=0$, equation
\eqref{E:6.9} reduces to \eqref{E:00:2.8}, and so the structure is
Sasakian.
\end{proof}


\subsection*{Acknowledgment}

The second author is grateful to University Grants Commission, New
Delhi (Ref. No.:20/12/2015(ii)EU-V) for financial support in the
form of Junior Research Fellowship.

\end{document}